\documentclass[12pt,reqno]{amsart}

\usepackage{amssymb}

\newtheorem{theorem}{Theorem}[section]

\newtheorem{lemma}[theorem]{Lemma}
\newtheorem{corollary}[theorem]{Corollary}

\theoremstyle{definition}

\newtheorem{remark}[theorem]{Remark}

\numberwithin{equation}{section}

\newcommand{\ZZ}{\mathbb{Z}}
\newcommand{\PP}{\mathbb{P}}
\newcommand{\cO}{\mathcal{O}}

\begin{document}

\title[Harder-Narasimhan filtration of the direct image]{On the Harder-Narasimhan filtration of the direct image 
of the structure sheaf}

\author[I. Biswas]{Indranil Biswas}

\address{Department of Mathematics, Shiv Nadar University, NH91, Tehsil
Dadri, Greater Noida, Uttar Pradesh 201314, India}

\email{indranil.biswas@snu.edu.in, indranil29@gmail.com}

\author[M. Kumar]{Manish Kumar}

\address{Statistics and Mathematics Unit, Indian Statistical Institute,
Bangalore 560059, India}

\email{manish@isibang.ac.in}

\author[A.J. Parameswaran]{A. J. Parameswaran}

\address{Kerala School of Mathematics, Kunnamangalam PO, Kozhikode, Kerala, 673571, India}

\email{param@ksom.res.in}

\subjclass[2010]{14H30, 14G17, 14H60}

\keywords{Direct image, semistability, Harder-Narasimhan filtration}

\begin{abstract}
We compute the Harder-Narasimhan filtration of vector bundles $f_*\cO_Y$ for certain finite morphisms $f\,
:\,Y\,\longrightarrow\, X$ and in some other cases.
\end{abstract}

\maketitle

\section{Introduction}

The Harder-Narasimhan filtration of a vector bundle $E$ on a smooth projective curve $X$ is a canonical filtration which is useful to 
study the vector bundles which may not be semistable. While the filtration is unique and has nice properties, it is very difficult to 
compute in general. In this note we compute explicitly the Harder-Narasimhan filtration in the following instances.

\begin{enumerate}
\item Let $E$ be a rank two vector bundle on $X$ and $Y$ an ample divisor on $\PP_X(E)$. Let $f$ be the finite morphism obtained 
by the restriction of the structure morphism $\PP_X(E)\,\longrightarrow\, X$ to $Y$. Lemma \ref{lem1} computes the Harder-Narasimhan
filtration of $f_*\cO_Y$ in terms of the Harder-Narasimhan filtration of $E$.
 
\item Let $Y$ be a complete intersection of codimension $r$ in $\PP^n$, and let $f:Y\,\longrightarrow\, \PP^{n-r}$ be the projection
morphism from a linear subspace of dimension $r$. Then Theorem \ref{main} does the computation for $f_*\cO_Y$.
\end{enumerate}

Let $\phi\,:\,X\,\longrightarrow\, \PP^1$ be a finite morphism with $X$ a smooth irreducible curve and $E$
a vector bundle on $X$ with vanishing first and zeroth cohomology. Then we observe that the
direct image $\phi_*E$ is the direct sum of copies of $\cO_{\PP^1}(-1)$ (Lemma \ref{lem-f}).

\section{Direct image for curves in a ruled surface}

Let $X$ be an irreducible smooth projective curve over an algebraically
closed field $k$. Take a vector bundle $E$ on $X$ of rank two. Let
\begin{equation}\label{e1}
\phi\,\, :\,\,{\mathbb P}(E) \,\, \longrightarrow\,\, X
\end{equation}
be the projective bundle parametrizing the one-dimensional quotients of the fibers of $E$. We have the tautological line bundle
${\mathcal O}_{{\mathbb P}(E)}(1)\, \longrightarrow\, {\mathbb P}(E)$ whose fiber over any
$z\, \in\, {\mathbb P}(E)$ is the quotient of $E_{\phi(z)}$ represented by $z$, where
$\phi$ is the projection in \eqref{e1}. Take an ample line bundle
\begin{equation}\label{e2}
L\,:=\, {\mathcal O}_{{\mathbb P}(E)}(n)\otimes \phi^*A \,=\, {\mathcal O}_{{\mathbb P}(E)}(1)^{\otimes n}\otimes \phi^*A 
\end{equation}
on ${\mathbb P}(E)$, where $A$ is a line bundle on $X$. Since $L$ is ample, its restriction to
any fiber of $\phi$ is ample, and this implies that $n\, \geq\, 1$.

Let $Y\,=\, {\rm Div}(s)\, \subset\, {\mathbb P}(E)$ be the divisor of a nonzero section
$s\, \in\, H^0({\mathbb P}(E),\, L)\setminus \{0\}$ (see \eqref{e2}); in other words, we have
$Y\, \in\, \big\vert L\big\vert$. Let
$$
f\, :=\, \phi\big\vert_Y\, :\, Y\, \longrightarrow\, X
$$
be the restriction of $\phi$ in \eqref{e1} to $Y$. We have the short exact sequence of coherent
sheaves on ${\mathbb P}(E)$
\begin{equation}\label{e3}
0\, \longrightarrow\, {\mathcal O}_{{\mathbb P}(E)}(-Y)\,=\, L^*\, \longrightarrow\,
{\mathcal O}_{{\mathbb P}(E)} \, \longrightarrow\, {\mathcal O}_Y \, \longrightarrow\, 0;
\end{equation}
the homomorphism $L^*\, \longrightarrow\, {\mathcal O}_{{\mathbb P}(E)}$ in \eqref{e3} sends
any $v\, \in\, L^*_x$, $x\, \in\,{\mathbb P}(E)$, to $v(s(x))\, \in\, k\,=\,
({\mathcal O}_{{\mathbb P}(E)})_x$.
Taking the direct image, under $\phi$, of the short exact sequence in \eqref{e3} we have the exact sequence
\begin{equation}\label{e4}
0 \, \longrightarrow\, \phi_*{\mathcal O}_{{\mathbb P}(E)}\,=\, {\mathcal O}_X
\, \longrightarrow\, \phi_*{\mathcal O}_Y \,=\, f_*{\mathcal O}_Y \, \longrightarrow\, 
R^1\phi_* {\mathcal O}_{{\mathbb P}(E)}(-Y)\, \longrightarrow\ 0,
\end{equation}
because $\phi_*{\mathcal O}_{{\mathbb P}(E)}(-Y)\,=\,0\,=\, R^1\phi_* {\mathcal O}_{{\mathbb P}(E)}$.

We will calculate
$$
R^1\phi_* {\mathcal O}_{{\mathbb P}(E)}(-Y)\, =\, R^1\phi_* L^*
$$
(see \eqref{e3} for the isomorphism).

Note that $R^1\phi_* {\mathcal O}_{{\mathbb P}(E)}(-Y)\, =\, R^1\phi_* L^* \,=\, 0$ if $n\, =\,1$
(see \eqref{e2} for $n$). Indeed, this follows immediately from the fact that $H^1({\mathbb P}^1_k,
\, {\mathcal O}_{{\mathbb P}^1_k}(-1))\,=\,0$. So we will assume that $n\, \geq\, 2$.

By the projection formula and \eqref{e2},
\begin{equation}\label{e5}
R^1\phi_* L^* \,=\, R^1\phi_*({\mathcal O}_{{\mathbb P}(E)}(-n)\otimes \phi^*A^*)\,=\,
(R^1\phi_*{\mathcal O}_{{\mathbb P}(E)}(-n))\otimes A^*.
\end{equation}

Let $T_\phi\,:=\, K^{-1}_{{\mathbb P}(E)}\otimes \phi^* K_X$ be the relative tangent bundle for the
projection $\phi$ in \eqref{e1}, where $K_X$ and $K_{{\mathbb P}(E)}$ are the canonical
line bundles of $X$ and ${\mathbb P}(E)$ respectively. We note that
$$
T_\phi\,=\, {\mathcal O}_{{\mathbb P}(E)}(2)\otimes \bigwedge\nolimits^2 \phi^* E^*\,=\,
{\mathcal O}_{{\mathbb P}(E)}(2)\otimes \phi^* \det E^* .
$$
Indeed, this follows immediately from the fact that $T_\phi$ is identified with the line bundle
${\rm Hom}({\mathcal O}_{{\mathbb P}(E)}(-1)\otimes\phi^*\det E,\, {\mathcal O}_{{\mathbb P}(E)}(1))$.
The relative version of Serre duality says that
$$
R^1\phi_*{\mathcal O}_{{\mathbb P}(E)}(-n)\,=\, (\phi_*(T^*_\phi\otimes {\mathcal O}_{{\mathbb P}(E)}(n)))^*
$$
$$
=\, (\phi_* ({\mathcal O}_{{\mathbb P}(E)}(n-2)\otimes \phi^* \det E))^*\,=\,
(\phi_* {\mathcal O}_{{\mathbb P}(E)}(n-2))^* \otimes \det E^*;
$$
the last isomorphism is given by the projection formula. This combined with \eqref{e5}, yields
$$
R^1\phi_* L^* \,=\, (\phi_* {\mathcal O}_{{\mathbb P}(E)}(n-2))^*\otimes (\det E^*)\otimes A^*\,=\,
\text{Sym}^{n-2}(E)^*\otimes A^* \otimes \det E^*
$$
$$
=\,\text{Sym}^{n-2}(E^*)\otimes A^*\otimes \det E^* ;
$$
this is because $\phi_* {\mathcal O}_{{\mathbb P}(E)}(j)\,=\, \text{Sym}^j(E)$ for all $j\, \geq\, 0$
(recall that $n\, \geq\, 2$ and $\text{Sym}^{0}(E^*) \,=\, {\mathcal O}_X$).
Substituting this in \eqref{e4} we get an exact sequence
\begin{equation}\label{e6}
0 \, \longrightarrow\, {\mathcal O}_X \, \longrightarrow\, f_*{\mathcal O}_Y \, \longrightarrow\, 
\text{Sym}^{n-2}(E^*)\otimes A^*\otimes \det E^* \, \longrightarrow\ 0,
\end{equation}
This implies that
\begin{equation}\label{e7}
(f_*{\mathcal O}_Y)/{\mathcal O}_X\,\,=\,\, \text{Sym}^{n-2}(E^*)\otimes A^*\otimes \det E^*\, .
\end{equation}

\begin{lemma}\label{lem1}
Assume that one of the following two holds:
\begin{enumerate}
\item The characteristic of $k$ is zero and $E$ is semistable.

\item The vector bundle $E$ is strongly semistable.
\end{enumerate}
Then $(f_*{\mathcal O}_Y)/{\mathcal O}_X$ is semistable if the characteristic of $k$ is zero, and
$(f_*{\mathcal O}_Y)/{\mathcal O}_X$ is strongly semistable if the characteristic of $k$ is positive.

If $E$ is not semistable, and
$S\, \subset\, E$ is the Harder--Narasimhan filtration of $E$, then
$$
(E/S)^{\otimes (2-n)}\otimes A^*\otimes \det E^* \, \subset\,
(E/S)^{\otimes (3-n)}\otimes E^*\otimes A^*\otimes \det E^*
$$
$$
\subset\, (E/S)^{\otimes (4-n)}\otimes {\rm Sym}^{2}(E^*)\otimes A^*\otimes \det E^*
\, \subset\, \cdots \, \subset\, (E/S)^{\otimes -2}\otimes{\rm Sym}^{n-4}(E^*)
\otimes A^*\otimes \det E^*
$$
$$
\subset\, (E/S)^*\otimes{\rm Sym}^{n-3}(E^*)\otimes A^*\otimes \det E^*
\, \subset\,{\rm Sym}^{n-2}(E^*)\otimes A^*\otimes \det E^*
$$
is the Harder--Narasimhan filtration of $(f_*{\mathcal O}_Y)/{\mathcal O}_X$.
\end{lemma}

\begin{proof}
If the characteristic of $k$ is zero, and $E$ is semistable, then the symmetric product 
$$\text{Sym}^{n-2}(E^*)\ =\ \text{Sym}^{n-2}(E)^*$$ is semistable \cite[p.~285, Theorem 3.18]{RR}.
On the other hand, if the characteristic of $k$ is positive, and $E$ is strongly semistable, then 
$\text{Sym}^{n-2}(E^*)$ is strongly semistable \cite[p.~288, Theorem 3.23]{RR}. Therefore, the
first statement follows from \eqref{e7}.

Let $A\, \subset\, B$ be the Harder--Narasimhan filtration of an unstable rank two vector bundle $B$ on $X$.
We will show that for any integer $m\, \geq\, 1$, the Harder--Narasimhan filtration of the symmetric
product $\text{Sym}^m(B)$ is the following:
\begin{equation}\label{f1}
A^{\otimes m}\, \subset\, A^{\otimes (m-1)}\otimes B \, \subset\, A^{\otimes (m-2)}\otimes\text{Sym}^2(B)\,
\subset\, \cdots\, \subset\, A\otimes\text{Sym}^{m-1}(B)\,\subset\, \text{Sym}^m(B).
\end{equation}
To prove that \eqref{f1} is the Harder--Narasimhan filtration, note that for
any $0\, \leq\, i\, \leq\, m-1$, the successive quotient $A^{\otimes i}\otimes\text{Sym}^{m-i}(B)/
(A^{\otimes (i+1)}\otimes\text{Sym}^{m-i-1}(B))$ in \eqref{f1} is the line bundle $A^{\otimes i}\otimes
(B/A)^{\otimes (m-i)}$. This implies the following three statements:
\begin{enumerate}
\item The successive quotient $A^{\otimes i}\otimes\text{Sym}^{m-i}(B)/(A^{\otimes (i+1)}\otimes
\text{Sym}^{m-i-1}(B))$ in \eqref{f1} is semistable for all $0\, \leq\, i\, \leq\, m-1$,

\item the inequality $$\text{degree}(A^{\otimes m})\ =\ m\cdot\text{degree}(A)\ > $$
$$
(m-1)\cdot \text{degree}(A)+ \text{degree}(B/A) \ =\ \text{degree}(A^{\otimes (m-1)}\otimes (B/A))$$
holds, because
$\text{degree}(A)\,> \, \text{degree}(B/A)$ (recall that $A\, \subset\, B$ is the Harder--Narasimhan
filtration of $B$), and

\item the inequality $$\text{degree}(A^{\otimes i}\otimes\text{Sym}^{m-i}(B)/(A^{\otimes (i+1)}\otimes
\text{Sym}^{m-i-1}(B)))\ = $$
$$i\cdot \text{degree}(A)+(m-i)\cdot \text{degree}(B/A)\ <\ (i+1)\text{degree}(A)+(m-i-1)\text{degree}(B/A)
$$
$$
=\ \text{degree}(A^{\otimes (i+1)}\otimes\text{Sym}^{m-i-1}(B)/(A^{\otimes (i+2)}\otimes
\text{Sym}^{m-i-2}(B)))$$
holds for all $0\, \leq\, i\, \leq\, m-2$, because
$\text{degree}(A)\,>\, \text{degree}(B/A)$.
\end{enumerate}
{}From these three statements it follows immediately that \eqref{f1} is the
Harder--Narasimhan filtration of $\text{Sym}^m(B)$ (see \cite[p.~16, Definition 1.3.2]{HL}).

Since \eqref{f1} is the Harder--Narasimhan filtration of $\text{Sym}^m(B)$,
the second statement of the lemma follows from \eqref{e7}.
\end{proof}

\begin{remark}
When $\dim X \, >\,1$, the first statement of Lemma \ref{lem1} remains valid. There is no change in the proof.
The second statement of Lemma \ref{lem1} remains valid if $\dim X\,=\, 2$.
But when $\dim X \, >\,2$, the second statement of Lemma \ref{lem1} needs to be modified. In this case, $(E/S)^*$
is a reflexive subsheaf of $E$ and it need not be locally free. Hence $((E/S)^*)^{\otimes j}$ may have torsion
part when $j\, >\, 1$. In the second statement of Lemma \ref{lem1},
$((E/S)^*)^{\otimes j}$ should be replaced by the torsionfree part of $((E/S)^*)^{\otimes j}$
\end{remark}

\begin{remark}
If $B$ is an unstable vector bundle of rank at least three, then the 
Harder--Narasimhan filtration of $\text{Sym}^m(B)$
 is --- in general --- much more complicated than \eqref{f1}. In fact, there
is no general formula to derive the Harder--Narasimhan filtration of $\text{Sym}^m(B)$
just from the Harder--Narasimhan filtration of $B$. The type of the Harder--Narasimhan filtration
of $\text{Sym}^m(B)$ depends --- in general --- on the degrees of the subsheaves occurring in the
Harder--Narasimhan filtration of $B$, when the rank of $B$ is at least three. This is unlike the case of
rank two where the Harder--Narasimhan filtration of $\text{Sym}^m(B)$ is given by the
Harder--Narasimhan filtration $A\,\subset\, B$ of $B$ in a way which does not depend on the degrees
of $A$ and $B/A$.
\end{remark}

\section{Intersection of two hypersurfaces}

Let $E$ be a vector bundle of rank three on $X$. As before,
\begin{equation}\label{e8}
\phi\,\, :\,\,{\mathbb P}(E) \,\, \longrightarrow\,\, X
\end{equation}
is the projective bundle parametrizing the one-dimensional quotients of the fibers of $E$,
and ${\mathcal O}_{{\mathbb P}(E)}(1)\, \longrightarrow\, {\mathbb P}(E)$ is the
tautological line bundle whose fiber over any $z\, \in\, {\mathbb P}(E)$ is the one-dimensional
quotient of $E_{\phi(z)}$ represented by $z$.

For $i\,=\, 1,\, 2$, take an ample line bundle
\begin{equation}\label{e9}
L_i\,\,:=\,\, {\mathcal O}_{{\mathbb P}(E)}(n_i)\otimes \phi^*A_i \,\,\longrightarrow\,\, {\mathbb P}(E),
\end{equation}
where $A_i$ is a line bundle on $X$. Also, take a nonzero section
\begin{equation}\label{e10a}
s_i\,\,\in\,\, H^0({\mathbb P}(E),\, L_i)\setminus \{0\}
\end{equation}
of $L_i$ in \eqref{e9} such that the map
\begin{equation}\label{e10}
f\,:=\, \phi\big\vert_{{\rm Div}(s_1)\cap {\rm Div}(s_2)}\,\, :\,\, Y\,:=\,
{\rm Div}(s_1)\cap {\rm Div}(s_2)\, \longrightarrow\, X
\end{equation}
is a finite morphism, where $\phi$ is the projection in \eqref{e8} and $s_i$ are the sections in
\eqref{e10a}; in \eqref{e10}, ${\rm Div}(s_i)$ denotes the divisor of $s_i$. We have the exact
sequence on ${\mathbb P}(E)$
\begin{equation}\label{e11}
0\, \longrightarrow\, L^*_1\otimes L^*_2\, \xrightarrow{\,\,\, \gamma_1\,\,}\,
L^*_1\oplus L^*_2\, \xrightarrow{\,\,\, \gamma_2\,\,}\, {\mathcal O}_{{\mathbb P}(E)}
\, \longrightarrow\, {\mathcal O}_Y \, \longrightarrow\,0,
\end{equation}
where the map $\gamma_1$ sends any $v\otimes w\, \in\, (L^*_1\otimes L^*_2)_x$, $x\, \in\, {\mathbb P}(E)$,
to $$(w(s_2(x))\cdot v,\, -v(s_1(x))\cdot w)\,\,\in\,\, (L^*_1\oplus L^*_2)_x,$$ and
the map $\gamma_2$ sends any $v\oplus w\, \in\, (L^*_1\oplus L^*_2)_x$, $x\, \in\, {\mathbb P}(E)$,
to $v(s_1(x))+ w(s_2(x))\,\in\,k\,=\, \left({\mathcal O}_{{\mathbb P}(E)}\right)_x$.

Let $\textbf{I}\,:=\, \gamma_2(L^*_1\oplus L^*_2)\, \subset\, {\mathcal O}_{{\mathbb P}(E)}$ be the image of
$\gamma_2$ in \eqref{e11}. The exact sequence in \eqref{e11} breaks into two short exact sequences:
\begin{align}
0\, \longrightarrow\, L^*_1\otimes L^*_2\, \xrightarrow{\,\,\, \gamma_1\,\,}\,
L^*_1\oplus L^*_2\, \xrightarrow{\,\,\, \gamma_2\,\,}\, \textbf{I}\, \longrightarrow\, 0\label{e12}\\
0\, \longrightarrow\, \textbf{I}\, \longrightarrow\, {\mathcal O}_{{\mathbb P}(E)}
\, \longrightarrow\, {\mathcal O}_Y \, \longrightarrow\,0.\label{e13}
\end{align}
The short exact sequence in \eqref{e13} produces the following long exact sequence of direct images:
\begin{align}
0\,=\, \phi_*\textbf{I} \, \longrightarrow\, \phi_*{\mathcal O}_{{\mathbb P}(E)}\,=\, {\mathcal O}_X
\, \longrightarrow\, f_*{\mathcal O}_Y\, \longrightarrow\, R^1\phi_*\textbf{I}
\,\longrightarrow\, R^1\phi_*{\mathcal O}_{{\mathbb P}(E)} \nonumber\\
\,=\, 0 \, \longrightarrow\, R^1f_*{\mathcal O}_Y\,=\,0
\, \longrightarrow\, R^2\phi_*\textbf{I}\, \longrightarrow\, R^2\phi_*{\mathcal O}_{{\mathbb P}(E)}\,=\, 0. \label{e14}
\end{align}
In \eqref{e14} we have used the fact that $H^1({\mathbb P}^2_k,\, {\mathcal O}_{{\mathbb P}^2_k})\,=\, 0$
(respectively, $H^2({\mathbb P}^2_k,\, {\mathcal O}_{{\mathbb P}^2_k})\,=\, 0$) to conclude that
$R^1\phi_*{\mathcal O}_{{\mathbb P}(E)} \,=\, 0$ (respectively, $R^2\phi_*{\mathcal O}_{{\mathbb P}(E)} \,=\, 0$);
also, $ R^1f_*{\mathcal O}_Y\,=\,0$ because $f$ is a finite map. From \eqref{e14} we conclude that
\begin{align}
R^2\phi_*\textbf{I}\, =\, 0, \label{e15}\\
(f_*{\mathcal O}_Y)/{\mathcal O}_X \, =\, R^1\phi_*\textbf{I}.\label{e16}
\end{align}

The short exact sequence in \eqref{e12} produces the following long exact sequence of direct images:
\begin{equation}\label{e17}
R^1\phi_*(L^*_1\oplus L^*_2)\, \longrightarrow\, R^1\phi_*\textbf{I} \, \longrightarrow\, R^2\phi_*(L^*_1\otimes L^*_2)
\, \xrightarrow{\,\,\, \gamma_{1*}\,\,}\, R^2\phi_* (L^*_1\oplus L^*_2) \, \longrightarrow\, R^2\phi_*\textbf{I}.
\end{equation}
In \eqref{e17}, we have $R^1\phi_*(L^*_1\oplus L^*_2)\,=\, 0$ because $H^1({\mathbb P}^2_k,\, {\mathcal O}_{{\mathbb P}^2_k}(m))
\,=\, 0$ for all $m\, \in\, {\mathbb Z}$. Therefore, using \eqref{e15} and \eqref{e16} in \eqref{e17} it is deduced that
\begin{equation}\label{e18}
(f_*{\mathcal O}_Y)/{\mathcal O}_X \, =\, R^1\phi_*\textbf{I}\,=\, {\rm kernel}(\gamma_{1*}).
\end{equation}

We will calculate $R^2\phi_*(L^*_1\otimes L^*_2)$ and $R^2\phi_*(L^*_1\oplus L^*_2)$
in \eqref{e17}. From \eqref{e9},
\begin{equation}\label{e19}
L^*_1\otimes L^*_2\,=\, {\mathcal O}_{{\mathbb P}(E)}(-n_1-n_2)\otimes \phi^*(A^*_1\otimes A^*_2),
\end{equation}
\begin{equation}\label{e19b}
L^*_1\oplus L^*_2\,=\, ({\mathcal O}_{{\mathbb P}(E)}(-n_1)\otimes \phi^* A^*_1)\oplus
({\mathcal O}_{{\mathbb P}(E)}(-n_2)\otimes \phi^* A^*_2).
\end{equation}
Let $K_\phi\,:=\, K_{{\mathbb P}(E)}\otimes (\phi^* K_X)^*$ be the relative canonical bundle for the
projection $\phi$ in \eqref{e8}. Note that
$$
K_\phi\,\,=\,\, {\mathcal O}_{{\mathbb P}(E)}(-3)\otimes \phi^*\det E.
$$
Hence using the relative version of Serre duality and \eqref{e19}, we have
\begin{equation}\label{e20}
R^2\phi_*(L^*_1\otimes L^*_2)\,\,=\,\, (\phi_*({\mathcal O}_{{\mathbb P}(E)}(n_1+n_2-3) \otimes
\phi^*(A_1\otimes A_2\otimes \det E)))^*
\end{equation}
$$
=\,\, (\text{Sym}^{n_1+n_2-3}(E)\otimes A_1\otimes A_2\otimes \det E)^*\,\,=\,\,
\text{Sym}^{n_1+n_2-3}(E^*)\otimes (A_1\otimes A_2\otimes \det E)^*,
$$
where the second isomorphism is given by the projection formula.
If $n_1+n_2\, <\, 3$, then we have $R^1\phi_*\textbf{I}\,=\, 0$ (use \eqref{e17}, \eqref{e20} and $R^1\phi_*(L^*_1\oplus L^*_2)\,=\, 0$).
Therefore, it is assumed that $n_1+n_2\, \geq\, 3$.

Similarly, using the relative version of Serre duality, and \eqref{e19b}, we have
\begin{equation}\label{e21}
R^2\phi_*(L^*_1\oplus L^*_2)\,\,=
\end{equation}
$$
(\phi_*({\mathcal O}_{{\mathbb P}(E)}(n_1-3) \otimes
\phi^*(A_1\otimes \det E)))^*
\oplus (\phi_*({\mathcal O}_{{\mathbb P}(E)}(n_2-3) \otimes \phi^*(A_2\otimes \det E)))^*
$$
$$
=\,\, (\text{Sym}^{n_1-3}(E)\otimes A_1\otimes \det E)^*\oplus (\text{Sym}^{n_2-3}(E)\otimes A_2\otimes \det E)^*
$$
$$
\,\,=\,\,
(\text{Sym}^{n_1-3}(E^*)\otimes (A_1\otimes\det E)^*)\oplus (\text{Sym}^{n_2-3}(E^*)\otimes (A_2\otimes\det E)^*);
$$
$\text{Sym}^j(E^*)\,=\, 0$ for $j\, <\, 0$.

Note that $s_i$ in \eqref{e10a} is a section of $\text{Sym}^{n_i}(E)\otimes A_i$. Using the isomorphisms in
\eqref{e20} and \eqref{e21}, the map $\gamma_1$ in \eqref{e17} is given by contractions using $s_1$ and $s_2$.
Now \eqref{e18} describes $(f_*{\mathcal O}_Y)/{\mathcal O}_X$.

Extending this to higher rank vector bundles on $X$ becomes cumbersome. In fact, the generalization of
the breaking up of \eqref{e11} into \eqref{e12} and \eqref{e13} becomes too involved to be
able to use them in the way it is done above.

\section{Complete intersection in projective space}

Let $H_1,\, \cdots,\, H_r$ be hypersurfaces in $\PP^n_k$ of degree $d_1,\,\cdots,\, d_r$ respectively
such that $X_i\,=\, H_1\cap H_2\ldots \cap H_i$ is a complete intersection of dimension $n-i$ for
any $1\,\le\, i\,\le\, r$. Let $L_1\,\subset\, L_2\,\subset\, \ldots\, \subset\, L_r$ be linear
subspace in the projective space $\PP^n_k$ such that
$\dim L_i\,=\,i-1$ and $L_i\cap X_i$ is empty. Let $$f_i\ :\ X_i\ \longrightarrow\ \PP^{n-i}_k$$
be the natural projection from $L_i$.

\begin{lemma}\label{l1}
 Let $X$ be $X_r$ and $f\,=\,f_r$ in the above notation. Then
$$f_*\cO_X\ = \ \oplus_j [\cO_{\PP^{n-r}_k}(-j)]^{a_j}$$ is a direct sum of line bundles of
nonpositive degrees.
\end{lemma}

\begin{proof}
By Horrocks' criterion (\cite{Horrocks}), it is enough to show that $H^m(\PP^{n-r}_k,(f_*\cO_X)(j))\,=\,0$ for all $j$ and for 
all $0\,<\, m \,<\, n-r$.

Note that $$H^m(\PP^{n-r}_k,\, (f_*\cO_X)(j))\,=\,H^m(X,\, \cO_X(j))\,=\,H^m(\PP^n_k,\, \cO_{\PP^n}(j))\,=\,0$$
for $1\,<\,m\,<\,n-r$ as $X$ is a complete intersection of dimension $n-r$.
For $j\,<\, 0$, we have $a_j\,=\,0$ because $H^0(\PP^{n-r}_k,\,f_*\cO_{X})$ is one dimensional so
$f_*\cO_{X}$ can't have a positive degree subbundle.
\end{proof}

Let
\begin{equation}\label{d1}
h_{i,j}\,=\,\dim(H^0(X_i,\,\cO_{X_i}(j)))
\end{equation}
for $0\,\le\, i\,\le\, r$ and $j\,\in\,\ZZ$ where we consider $X_0$ to be $\PP^n$. Note
that $h_{i,j}\,=\,0$ for $j\,<\,0$ and $0\,\le\, i\,\le\, r$. Moreover $h_{0,j}=\binom{n+j}{n}$ for $j\ge 0$.

\begin{theorem}\label{main}
The $h_{i,j}$'s and $a_j$'s (see \eqref{d1} and Lemma \ref{l1}) satisfy the following relations:
\begin{enumerate}
\item $h_{i,j}\,=\,h_{i-1,j}-h_{i-1,j-d_i}$ for $i\,>\, 0$ and all $j$.

\item $\sum_{j=0}^m \binom{n-r+m-j}{m-j}a_j\,=\,h_{r,m}$ for all $m\,\ge\, 0$.
\end{enumerate}
\end{theorem}

\begin{proof}
 For the (1), we use the following short exact sequence of sheaves
$$0\,\longrightarrow\, \cO_{X_{i-1}}(j-d_i)\,\longrightarrow\, \cO_{X_{i-1}}(j) \,\longrightarrow\,
g_{i*}\cO_{X_i}(j)\,\longrightarrow\, 0,$$
where $g_i\,:\,X_i\,\longrightarrow\, X_{i-1}$ is the closed immersion.
Since $H^1(X_{i-1},\,\cO_{X_{i-1}}(j-d_i))\,=\,0$ and $H^0(X_i,\,\cO_{X_i}(j))\,=\,H^0(X_{i-1},\,
g_{i*}\cO_{X_i}(j))$, (1) follows.

 For (2), let $g\,:\,X_r\,\longrightarrow\, \PP^n_k$ be the closed immersion. Note that
 \begin{align*}
 H^0(X_r,\,\cO_{X_r}(m)) &\,=\,H^0(X_r,\,\cO_{X_r}\otimes g^*\cO_{\PP^n_k}(m))\\
 &\,=\, H^0(X_r,\,\cO_{X_r}\otimes f^*\cO_{\PP^{n-r}}(m))\ ~~ (\text{as }\, g^*\cO_{\PP^n_k}(m)\, \cong
\,f^*\cO_{\PP^{n-r}_k}(m) ) \\
&\,=\, H^0(\PP^{n-r}_k,\,f_*(\cO_{X_r}\otimes f^*\cO_{\PP^{n-r}_k}(m))) \\
&\,=\, H^0(\PP^{n-r}_k,\,f_*(\cO_{X_r})\otimes \cO_{\PP^{n-r}_k}(m))\ ~~~(Projection ~formula)\\
&\,=\,H^0(\PP^{n-r}_k,\,\oplus_{j\ge 0} [\cO_{\PP^{n-r}_k}(-j)]^{a_j}\otimes \cO_{\PP^{n-r}_k}(m)) ~~~ (Lemma~3.1)\\
&\,=\,\oplus_{j\ge 0} H^0(\PP^{n-r}_k,\, [\cO_{\PP^{n-r}}(m-j)]^{a_j}).\\
 \end{align*}
 Since $\dim (H^0(\PP^{n-r}_k,\, \cO_{\PP^{n-r}}(m-j)))\,=\,\binom{n-r+m-j}{m-j}$ for $j\,\le\, m$ and
it is zero otherwise, and $h_{r,m}\,=\,\dim( H^0(X_r,\,\cO_{X_r}(m)))$, we obtain (2).
\end{proof}

\begin{corollary}\label{hypersurface}
Let $X$ be a curve of degree $d$ in $\PP^2_k$, and let $f\,:\,X\,\longrightarrow\, \PP^1_k$ be the
projection from a point in $\PP^2_k$ outside $X$. Then 
$$f_*\cO_X\ =\ \cO_{\PP^1_k}\oplus\cO_{\PP^1_k}(-1)\oplus\ldots\oplus\cO_{\PP^1_k}(-d+1).$$
\end{corollary}

\begin{proof}
Substitute $r\,=\,1$, $n\,=\,2$ and $d_1\,=\,d$ in Theorem \ref{main}. Then 
 \[ h_{1,j}\,=\,h_{0,j}-h_{0,j-d}\,=\,\begin{cases}
\binom{j+2}{2}, &\, \text{ if } 0\,\le\, j\,<\,d;\\
\binom{j+2}{2}-\binom{j-d+2}{2}, &\, \text{ if } j\,\ge\, d.
\end{cases}
\]

Moreover $a_0\,=\,h_{1,0}\,=\,1$ and for $m\,\ge\, 1$, we have $$(m+1)a_0+ma_1+(m-1)a_2+\ldots+1a_m
\ =\ h_{1,m}.$$ Solving for $a_m$, we get
$a_m\,=\,h_{1,m}-((m+1)a_0+ma_1+(m-1)a_2+\ldots+2a_{m-1})$. Taking $m\,<\,d$ and using induction on $m$,
$$a_m\ =\ (m+2)(m+1)/2+((m+1)+m+(m-1)+\ldots +2)\ =\ 1.$$ Also for $m\,\ge\, d$, $$a_m\ =\ 0$$ can be
verified similarly or by observing that $f_*\cO_X$ is of rank $d$.
\end{proof}

\begin{corollary}
Let $X$ be a complete intersection curve of multi-degree $(d_1,\, d_2)$ with $d_1\,\le\, d_2$ in $\PP^3_k$, and let
$f\,:\,X\,\longrightarrow\, \PP^1_k$ be the projection from a line in $\PP^3_k$ not intersecting $X$.
Then $f_*\cO_X\,=\,\oplus_{j=0}^{d_1+d_2-1}\cO_{\PP^1}(-j)^{a_j}$, where $$(a_0,\,a_1,\,\cdots)
\,=\,(1,\,2,\,\cdots,\,d_1,\,d_1,\,\cdots,\,d_1,\,d_1-1,\,d_1-2,\,\cdots,\,1)$$
with $d_1$ repeating $d_2-d_1+1$ times.
\end{corollary}

\begin{proof}
Let $X_1$ be a surface of degree $d_1$ in $\PP^3_k$, and $X\,\subset\, X_1$.
Set $r\,=\,2$ and $n\,=\,3$ in Theorem \ref{main}. We have the following:
 \[ h_{1,j}\,=\,h_{0,j}-h_{0,j-d_1}\,=\,\begin{cases}
\binom{j+3}{3}, &\, \text{ if }\,\, 0\,\le\, j\,<\,d_1,\\
\binom{j+3}{3}-\binom{j-d_1+3}{3}, &\, \text{ if }\,\, j\,\ge\, d_1;
\end{cases}
\]
 
 \[ h_{2,j}\,=\,h_{1,j}-h_{1,j-d_2}\,=\,\begin{cases}
\binom{j+3}{3}, &\text{ if }\, 0\,\le\, j\,<\,d_1,\\
\binom{j+3}{3}-\binom{j-d_1+3}{3}, &\text{ if }\, d_1\,\le\, j \,<\, d_2,\\
\binom{j+3}{3}-\binom{j-d_1+3}{3}-\binom{j-d_2+3}{3} &\text{ if }\, d_2\le j < d_2+d_1,\\
\binom{j+3}{3}-\binom{j-d_1+3}{3}-\binom{j-d_2+3}{3}-\binom{j-d_2-d_1+3}{3} &\text{ if }\, j \,\ge\, d_2+d_1.
\end{cases}
\]
 
 Again by Theorem \ref{main} (2), $a_0\,=\,1$, and for $m\,\ge\, 1$,
$$(m+1)a_0+ma_1+(m-1)a_2+\ldots+1a_m\ =\ h_{2,m}.$$
 For $m\,<\,d_1$, by induction on $m$, 
$$a_m\, =\, (m+3)(m+2)(m+1)/6-[1(m+1)+2m+3(m-1)+\ldots+m(m-m+2)]\,=\,m+1,$$
since $\sum_{i=1}^{m+1}i(m+2-i)\,=\,\binom{m+3}{3}$.
 Again using this equation, $$a_{d_1}\ =\ \binom{d_1+3}{3}-1-[1(d_1+1)+2d_1+\ldots d_12]\ =\ d_1.$$

For $d_1\,<\, m \,<\,d_2$,
\begin{align*}
a_{m} &=\binom{m+3}{3}-\binom{m-d_1+3}{3}-\left[\sum_{i=1}^{d_1}i(m+2-i)+d_1\sum_{i=d_1+1}^{m}(m+2-i)\right]\\
&=\left[\sum_{i=d_1+1}^{m+1}i(m+2-i)\right]-\left[\sum_{i=1}^{m-d_1+1}i(m-d_1+2-i)\right]-d_1\left[\sum_{i=d_1+1}^{m}(m+2-i)\right]\\
&=\left[\sum_{i=d_1+1}^{m}(i-d_1)(m+2-i)\right]+m+1-\left[\sum_{i=1}^{m-d_1+1}i(m-d_1+2-i)\right]\\
&=d_1.
\end{align*}

Similar calculation for $d_2\,\le \,m\,\le\, d_2+d_1$ gives $a_m\,=\,d_1+d_2-m$.
\end{proof}

\begin{remark}
 Though the above computation has been done for curves, the formula for $a_m$'s depend only on
the codimension of the complete intersection. In other words, if $X$ is a hypersurface of degree $d$
in $\PP^n_k$, then its Harder-Narasimhan filtration is same as that given in Corollary \ref{hypersurface}.
\end{remark}

\begin{remark}
The same computation can be done for higher codimension complete intersection as well using
Theorem \ref{main} to obtain a pattern for its Harder-Narasimhan filtration. 
\end{remark}

\section{A property of direct images}

Let $X$ be an irreducible smooth projective curve over an algebraically closed field $k$, and
let $\varphi\, :\, X\, \longrightarrow\, {\mathbb P}^1_k$ be a finite morphism. Take a vector 
bundle $E$ on $X$ such that
\begin{equation}\label{z1}
H^0(X,\, E)\ =\ 0\ =\ H^1(X,\, E).
\end{equation}

\begin{lemma}\label{lem-f}
The direct image $\varphi_*E$ is a direct sum of copies of the line bundle
${\mathcal O}_{{\mathbb P}^1_k}(-1)$. In particular, $\varphi_*E$ is semistable.
\end{lemma}

\begin{proof}
A theorem of Grothendieck says that any vector bundle on ${\mathbb P}^1_k$ is a direct sum of line bundles
(\cite{Grothendieck-splitting}). Let
$$
\varphi_*E \ =\ \bigoplus_{i=1}^r L_i
$$
be a decomposition of $\varphi_* E$ into a direct sum of line bundles. We have
$H^j (X,\, E)\,=\, H^j({\mathbb P}^1_k,\, \varphi_*E)$ for $j\,=\, 0,\, 1$ because $\varphi$ is a finite
morphism. Consequently, the given in \eqref{z1} implies that
\begin{equation}\label{z2}
H^j({\mathbb P}^1_k,\, L_i)\,=\, 0
\end{equation}
for $j\,=\, 0,\, 1$ and $1\, \leq\, i\, \leq\, r$. From \eqref{z2} it follows immediately that
$L_i\,=\, {\mathcal O}_{{\mathbb P}^1_k}(-1)$.
\end{proof}

\section*{Acknowledgements}

We thank the referee for useful comments. The first-named
author is partially supported by a J. C. Bose Fellowship (JBR/2023/000003).


\begin{thebibliography}{ZZZZZ}

\bibitem{Grothendieck-splitting} A. Grothendieck, Sur la classification des fibr\'es holomorphes sur la 
sph\`ere de Riemann, {\it Amer. J. Math.} {\bf 79} (1957), 121--138.

\bibitem{Horrocks} G. Horrocks, Vector bundles on the punctured spectrum of a local ring, {\it Proc. London 
Math. Soc.} {\bf 14} (1964), 689--713.

\bibitem{HL} D. Huybrechts and M. Lehn, {\it The geometry of moduli spaces of sheaves}, Aspects
of Mathematics, E31, Friedr. Vieweg~\&~Sohn, Braunschweig, 1997.

\bibitem{RR} S. Ramanan and A. Ramanathan, Some remarks on the instability flag,
{\it Tohoku Math. Jour.} {\bf 36} (1984), 269--291.

\end{thebibliography}
\end{document}